\documentclass{amsart}
\usepackage{amsfonts}
\usepackage{amsmath}
\usepackage{amssymb}
\usepackage{amsthm}
\usepackage{multirow}
\usepackage{graphics}
\usepackage{graphicx}
\usepackage{float}
\usepackage{adjustbox}
\usepackage{enumerate}

\newcommand{\N}{\mathbb{N}}
\newtheorem{theorem}{Theorem}[section]

\newtheorem{lemma}[theorem]{Lemma}
\newtheorem{corollary}[theorem]{Corollary}
\newtheorem{remark}[theorem]{Remark}
\newtheorem{example}[theorem]{Example}

\theoremstyle{definition}
\newtheorem{definition}[theorem]{Definition}

\title{On a new formula for the number of unrestricted partitions}
\author[Hemar Godinho and Jos\'e Pl\'inio O. Santos]{} 
\thanks{$^{*}$ Corresponding Author}
\begin{document}
\maketitle
\centerline{\scshape Hemar Godinho$^{*}$}
\medskip
{\footnotesize
 \centerline{Departamento de Matem\' atica, Universidade de Bras\'{i}lia,  }
   \centerline{Bras\'{i}lia-DF, 70910-900, Brazil}
   \centerline{hemar@unb.br}
} 
\medskip
\centerline{\scshape Jos\'e Pl\'inio O. Santos}
\medskip
{\footnotesize
 \centerline{Departamento de Matem\' atica Aplicada - IMECC, Universidade de Campinas, }
   \centerline{Campinas-SP,  13083-859, Brazil}
   \centerline{josepli@ime.unicamp.br}
\begin{abstract}

 In this paper we  present a new formula  for the number of unrestricted partitions of $n$. We do this by introducing a 
  correspondence between the number of  unrestrited partitions of $n$  and the number of non-negative  solutions of systems of two equations, involving natural  numbers in the interval (1 $,n^{2}$). 
  
\end{abstract}
\section{Introduction}

The relation between partitions and two-line matrices dates back to 1900 with the work of Frobenius\cite{F}. These ideas were further developed by Andrews\cite{A}, who presented a relation between these matricial representations and the elliptic theta functions. A new approach was introduced in Mondek-Ribeiro-Santos\cite{M},  also relating partitions and two-line matrices, but now adding an important feature to this matricial representation, since the conjugate of the partition can also be read from the matrix representation.  This relation is described in  details in Brietzke-Santos-Silva\cite{B11, B13}. This theory was extended in Matte-Santos\cite{Ma}, where it is presented a correspondence between these matrices, paths in the Cartesian plane  and partitions into distinct odd parts all greater than one. These ideas are described in the next section.

Our starting point is the following important  theorem  proved in  \cite{M}.
\begin{theorem}\label{unr}
The number of unrestricted partitions of $n$ is equal to the number of two-line matrices 
\begin{equation}\label{mat}
\left ( \begin{array}{llll}
c_{1} & c_{2} & \cdots & c_{s} \\
d_{1}&d_{2} & \cdots & d_{s} 
\end{array}\right ),  \mbox{ with} \;\; c_{i},d_{j}\in \N\cup\{0\}, \;\;\mbox{for}\;\; 1\leq i,j\leq s.
\end{equation}
 such that 
\begin{equation}\label{C}
c_{s} = 0,\; d_{s}\neq 0, \;\; c_{j}= c_{j+1} + d_{j+1}, \;\;\mbox{for}\;\; \jmath = 1,2,\ldots, s-1, 
\end{equation}
and the sum of all entries is equal to  $n$.
\end{theorem}

Let $\mathcal{M}$ be the set of all matrices given in \eqref{mat} with the restrictions given in \eqref{C}. Theorem \ref{unr} above establishes an 1-1 correspondence between $\mathcal{M}$ and the set of all unrestricted partitions. In this  bijection,   the number $s$ of columns of the matrix in $\mathcal{M}$ corresponds to the
number of parts of the partition. Consider, for example, the partition $\lambda = (6, 5, 2, 2)$ of 15. This partition is associated to a $2\times4$ matrix $A\in \mathcal{M}$  in the following way: we have no choice for the
fourth column, but to pick $c_4= 0$ and $d_4 = 2$. Since $c_3$ must be 2 and the entries
of the third column must add up to 2, then $d_3$ = 0. By the same argument, we
must have $c_2 = 2$ and $d_2 = 3$. Also, $c_1 = 5$ and $d_1 = 1$. To go from the matrix to the partition, we only have to add the entries in each column, hence the correspondence is
\begin{center} 
$$\lambda=6+5+2+2 \;\; \longleftrightarrow \;\; \left( \begin{array}{cccc} 5&2&2&0 \\ 1&3&0&2 \end{array} \right)$$
\end{center} 

In the example below we present explicitly the correspondence between matrices of $\mathcal{M}$ and the unrestricted partitions of $n=5$.
\begin{example}
	\label{ex_matrices_6}
	For $n=5$ we have $p(5)=7$, and so there are $7$ matrices in $\mathcal{M}$ satisfying Theorem \ref{unr}, shown in Table \ref{table_ex_p(5)}.
	\begin{table}[H]
		\centering
		{\small
			\begin{tabular}{cc|cc}
				\hline\noalign{\smallskip}
				\multirow{2}{*}{Partition of $5$} & \multirow{2}{*}{Matrix in $\mathcal{M}$}  & \multirow{2}{*}{Partition of $5$} & \multirow{2}{*}{Matrix in $ \mathcal{M}$}\\
				&  \\
				\noalign{\smallskip}\hline\noalign{\smallskip}
				\vspace{0.1cm}  $(1,1,1,1,1)$ &  $\left( \begin{array}{cccccc} 1&1&1&1&0 \\ 0&0&0&0&1 \end{array} \right)$ & $(3,2)$ &  $\left( \begin{array}{cc} 2&0 \\ 1&2 \end{array} \right)$ \\
				\vspace{0.1cm}  $(2,1,1,1)$ &  $\left( \begin{array}{ccccc} 1&1&1&0 \\ 1&0&0&1 \end{array} \right)$ & $(4,1)$ &  $\left( \begin{array}{ccc} 1&0 \\ 3&1 \end{array} \right)$ \\ 
				\vspace{0.1cm}  $(2,2,1)$ &  $\left( \begin{array}{cccc} 2&1&0 \\ 0&1&1 \end{array} \right)$  &  $(5)$ &  $\left( \begin{array}{cc} 0 \\ 5 \end{array} \right)$ \\ 
				\vspace{0.1cm}  $(3,1,1)$ &  $\left( \begin{array}{cccc} 1&1&0 \\ 2&0&1 \end{array} \right)$  &  &   \\ 
				\noalign{\smallskip}\hline
			\end{tabular}
		}
		\caption{\footnotesize{Table for Example \ref{ex_p(6)}}}
		\label{table_ex_p(5)}
	\end{table} 
	\label{ex_p(6)}
\vspace{-0.5cm}
\end{example}
 In this paper we want to present a new  correspondence between the number of  unrestrited partitions of $n$  and the number of non-negative  solutions of systems of two equations, involving natural  numbers in the interval (1 $,n^{2}$). In order to reach this goal we need first to present  a procedure described in \cite{Ma}, relating unrestricted partitions with paths in the cartesian plane, and partitions of numbers in the interval  $(1 ,n^{2})$  into distinct odd parts all greater than 1. Then we show that these numbers can also be partitioned into squared parts, and finally present the relation of these partitions with systems of two equations.
 
\section{The Path Procedure}

We associate each matrix of $\mathcal{M}$, with at least two columns, to a path built through the Cartesian plane, connecting the point $Q=(\sum_{i=1}^{s} d_i,\sum_{i=1}^{s-1} c_i)$ (for $c_{s}=0$) in the line $x+y=n$ to the origin $(0,0)$. We choose the second line of the matrix to be associated to the $x$-axis, and the first line to be associated to the $y$-axis. 
The path consists of  the following sequence of points:
\begin{equation}\label{path}
\begin{array}{l}
Q=(\displaystyle\sum_{i=1}^{s} d_i,\sum_{i=1}^{s-1} c_i)\rightarrow (\sum_{i=1}^{s-1} d_i,\sum_{i=1}^{s-1} c_i)\rightarrow (\sum_{i=1}^{s-1} d_i,\sum_{i=1}^{s-2} c_i)\rightarrow \cdots \\
 \rightarrow (d_{1}+d_{2}, c_{1}+c_{2}) \rightarrow (d_1+d_2,c_1) \rightarrow (d_1,c_1)\rightarrow (d_1,0)\rightarrow (0,0). \end{array}
\end{equation}

\begin{example}
	For the partition   $(2,2,1,1)$ of $n=6$ we associate  the matrix
	$$M=\left( \begin{array}{cccc} 2&1&1&0 \\ 0&1&0&1 \end{array} \right)$$
and the path  $(2,4)\rightarrow (1,4)\rightarrow (1,3)\rightarrow (1,2)\rightarrow (0,2) \rightarrow (0,0).$
The Figure below describes the process of obtaing the path for this particular matrix $M$.
	\begin{figure}[!htb]
		\centering
		\includegraphics[width=6cm]{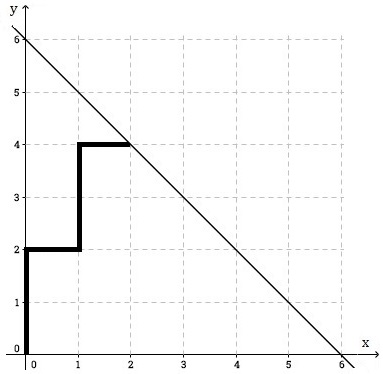}
		\caption{\footnotesize{Illustration for Example \ref{ex_partition_path}}}
		\label{fig_ex_partition_path}
	\end{figure}
	\label{ex_partition_path}
\end{example}

Now we reflect the path through the line $x+y=n$ and create a partition of a number $m$  into distinct odd parts all greater than 1, by taking hooks of the following sizes:
\begin{equation}\label{hooks}
\begin{array}{ccl}
\lambda_{1} & = &2(n-d_{1}) )-1, \\
\lambda_{2} & = & 2((n-d_{1})  -1)-1,   \\
&\vdots & \\
\lambda_{c_{1}} & = &  2((n-d_{1})-(c_{1} -1))-1,  \\
\lambda_{c_{1} + 1} & = & 2((n-d_{1})- d_{2} - c_{1})-1,  \\
&\vdots & \\
\lambda_{c_{1}+ c_{2}} & = & 2((n-d_{1})- d_{2} - c_{1} - (c_{2}-1))-1,  \\
\lambda_{c_{1}+ c_{2} + 1} & = & 2((n-d_{1})- d_{2} - d_{3} -  c_{1} - c_{2})-1, \\
&\vdots & \\
\lambda_{c_{1}+ c_{2} + \cdots + c_{s-1}} & = & 2((n-d_{1})- (\sum_{i=2}^{s-1} d_{i}) -\cdots  - (\sum_{j=1}^{s-1} c_{j}) +1)-1,
\end{array}\end{equation}
in this case  $m = \lambda_{1}+\cdots + \lambda_{c_{1}+\cdots + c_{s-1}}$.  This procedure of associating to a matrix a number $m$ partitioned into distinct odd parts all greater than one is called {\it The Path Procedure}. 
\begin{remark}\label{m<n2}
Since these hooks are all inside a square of side length equals to $n$,   for any  $\jmath$,  we have  $\lambda_{j}\leq 2n-1$ and $m\leq n^{2}-1$.
\end{remark}

\begin{example}
	For the matrix
	$M=\left( \begin{array}{cccc} 2&1&1&0 \\ 0&1&0&1 \end{array} \right), $
	associated to the partition $(2,2,1,1)$ of $n=6$, the hooks given by the reflection of the path through the line $x+y=6$ provide the parts $\lambda_{1}=11$,   $\lambda_{2}=9$,  $\lambda_{3}=5,$  and  $\lambda_{4}=3$ (see \eqref{hooks} above), a partition of $m=28$. Figure \ref{fig_ex_partition_hooks} below helps to understand the process.
	\begin{figure}[!htb]
		\centering
		\includegraphics[width=6cm]{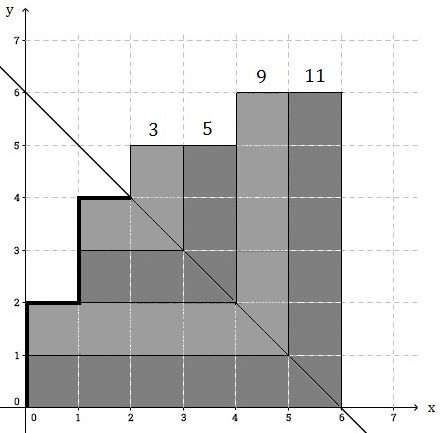}
		\caption{\footnotesize{Illustration for Example \ref{ex_partition_hooks}}}
		\label{fig_ex_partition_hooks}
	\end{figure}
	\label{ex_partition_hooks}
\end{example}
\begin{remark}\label{R1}
A closer look at the expressions of the parts $\lambda_{j}$ in \eqref{hooks} reveals that in all of them we have the term $(n-d_{1})$,  which is equal to $2c_{1}+c_{2}+\cdots + c_{s-1}$, according to the restritions stated in \eqref{C}. Therefore the term $d_{1}$ plays no role in the procedure of obtaining these parts. On the other hand, it is clear that each matrix associated to a different partition of a fixed number  $n$ generates a different path from the line $x+y=n$ to the origin $(0,0)$, which in turn  leads to a  different partition into distinct odd parts. These comments are illustrated in the two examples below. The  example \ref{ex_6_4} highlights the fact that matrices of $\mathcal{M}$ which only differs from  each other at the $d_{1}$ entry produces the same partitions into distinct odd parts, and example \ref{ex-p(5)}  presents the relation between the matrices associated to the partitions of $n=5$ and the partitions into distinct odd parts originated from the {\it Path Procedure}.
\end{remark}
\begin{example}
\label{ex_6_4}
Let us take the partition $(4,1,1)$ of $6$ and the partition $(2,1,1)$ of $4$. The matrices associated to them are, respectively,
$$M_{1}=\left( \begin{array}{ccc} 1&1&0 \\ 3&0&1 \end{array} \right) \mbox{ and }\; M_{2}=\left( \begin{array}{ccc} 1&1&0 \\ 1&0&1 \end{array} \right).$$

The paths that these matrices originate are different, although both paths induce the same hooks and, therefore, the same partition into distinct odd parts, as shown in Figure \ref{fig_ex_partition_hooks_6_4}.
\begin{figure}[!htb]
\centering
\includegraphics[width=12cm]{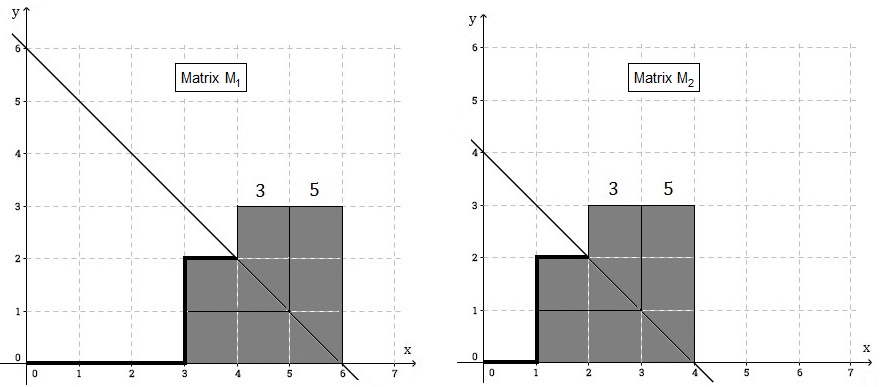}
\caption{\footnotesize{Illustration for Example \ref{ex_partition_hooks_6_4}}}
\label{fig_ex_partition_hooks_6_4}
\end{figure}
\label{ex_partition_hooks_6_4}
\end{example}

\begin{example}\label{ex-p(5)}
	 In the table below we present the result of the application of the {\it Path Procedure} to  the matrices  listed in   Table \ref{table_ex_p(5)}, related to the partitions of $n=5$.
	\begin{table}[H]
		\centering
		{\small
			\begin{tabular}{cc|cc}
				\hline\noalign{\smallskip}
				\multirow{2}{*}{Matrix in $\mathcal{M}$}  & \multirow{2}{*}{ \begin{tabular}{@{}c@{}} Partition into \\ [3pt] distinct odd parts \end{tabular}}  & \multirow{2}{*}{Matrix in $\mathcal{M}$}  & \multirow{2}{*}{ \begin{tabular}{@{}c@{}} Partition into \\ [3pt] distinct odd parts \end{tabular}} \\
				&  \\
				\noalign{\smallskip}\hline\noalign{\smallskip}
				\vspace{0.1cm}     $M_{1}=\left( \begin{array}{cccccc} 1&1&1&1&0 \\ 0&0&0&0&1 \end{array} \right)$& (9,7,5,3) &   $M_{5}=\left( \begin{array}{cc} 2&0 \\ 1&2 \end{array} \right)$ & (7,5) \\
				\vspace{0.1cm}   $M_{2}=\left( \begin{array}{ccccc} 1&1&1&0 \\ 1&0&0&1 \end{array} \right)$ & (7,5,3) &   $M_{6}=\left( \begin{array}{ccc} 1&0 \\ 3&1 \end{array} \right)$ & (3) \\ 
				\vspace{0.1cm}  $M_{3}=\left( \begin{array}{cccc} 2&1&0 \\ 0&1&1 \end{array} \right)$& (9,7,3)  &   $M_{7}=\left( \begin{array}{cc} 0 \\ 5 \end{array} \right)$ & $\emptyset$ \\ 
				\vspace{0.1cm}    $M_{4}=\left( \begin{array}{cccc} 1&1&0 \\ 2&0&1 \end{array} \right)$  & (5,3)&  &   \\ 
				\noalign{\smallskip}\hline
			\end{tabular}
		}
		\caption{\footnotesize{Table for Example \ref{ex-p(5)}}}
		\label{table_ex_p(5-1)}
	\end{table} 
\end{example}

In light of the comments above, especially the ones in Remark \ref{R1}, we may define a function $P:\mathcal{M} \longrightarrow \N$, defined as $P(M)=m$, where $m$ is obtained by the application of the Path Procedure to the matrix $M$. From Table \ref{table_ex_p(5-1)}, we have that 
$$P(M_{1})= 24, \; P(M_{2})=15, \;P(M_{3})= 19, \; P(M_{4})=8, \; P(M_{5})= 12 , \;P(M_{6})=3. $$
These numbers are all distinct, but clearly this happens only for small numbers otherwise the number of unrestricted partitions of $n$ would be limited by $n^2$. Observe that if we take the matrices
\begin{equation}\label{conv}
A=\left( \begin{array}{cccc} 4&1&1&0 \\ 0&3&0&1 \end{array} \right) \mbox{ and }\; B=\left( \begin{array}{ccc} 3&3&0 \\ 0&0&3 \end{array} \right),
\end{equation}
we obtain $P(A)= P(B)=72$ and it is easy to verify that $5,6,9,10 \not \in Im(P)$. This shows that the function $P$ is neither injective nor surjective,  in fact there are infinitly many values of $m\in\N$ not in the image of $P$, as we shall see shortly. If $m\in Im(P)$, let us call the size of the set $P^{-1}(m)$ the frequency of $m$, and  denote it by $f(m)$. In \cite{Ma} one can find tables for the frequencies corresponding to all the partitions of $4, 5, 8$ and $21$.  In particular, when considering  $n=21$, they have obtained $f(324)=8$, $f(291)=9$ and $f(312)=10$. The number $f(m)$ is always finite, for it is smaller than the number of partitions of $m$ into distinct odd parts.

\section{$t$-squared Partitions}
\begin{definition}\label{t-sqrt}
Let $m\in \N$. We say that $m$ admits a $t$-{\it squared partition} if we can partition  $m$ into $2t+1$ squared parts as
\begin{equation}\label{part}
m = b^2 +c_{1}^{2}+c_{2}^{2}+\cdots + c_{t}^{2} + c_{1}^{2}+ c_{2}^{2}+\cdots + c_{t}^{2},
\end{equation}
where $b= c_{1}+c_{2}+\cdots + c_{t}$ and $c_{1},\ldots,c_{t}\in\N$.
\end{definition}
It follows trivially from the definition that $m=b^{2} + 2a$, where $a= c_{1}^{2} + \cdots + c_{t}^{2}$. As an example we have that $83$ admits a $3$-squared partition since  
$$ 83= (3+2+2)^2 +2(3^2+2^2+2^2).$$

\begin{theorem}\label{Teo-sq}
 Let $m\in\N$. Then $m\in Im(P)$ if, and only if, $m$ admits a $t$-squared partition.
\end{theorem}
\begin{proof}
Let $M\in \mathcal{M}$,  written as   $M=\left ( \begin{array}{lllll} c_{1} & c_{2} & \cdots &c_{s-1}& 0 \\ 0&d_{2} & \cdots &d_{s-1}& d_{s} 
\end{array}\right ),$  such that $P(M)=m$. Once we have the  matrix $M$ we  associate to it the  path from $Q=(x_{0},y_{0})$ to the origin, where $x_{0}=\sum_{i=2}^{s}d_{i}$ and $y_{0}=\sum_{j=1}^{s-1} c_{j}$ (see \eqref{path}). Observe that, starting at $Q$ we obtain the next point moving horizontally to the left $d_{s}$ units. To proceed to the next point we move vertically downwards $c_{s-1}$ units. Continuing this process we obtain the following points by sequence of horizontal moves to the left  $d_{j}$ units and vertical moves downwards $c_{j}$ units. Eventually, after moving $d_{2}$ units horizontally to the left we reach the point $(0,c_{1})$, and  after moving $c_{1}$ units vertically downwards we reach the origin. It is clear from this construction that there exists  an 1-1 correpondence between this  path and matrix $M$.  The figure below illustrates this relationship.
\begin{figure}[!htb]
\centering
\includegraphics[width=12cm]{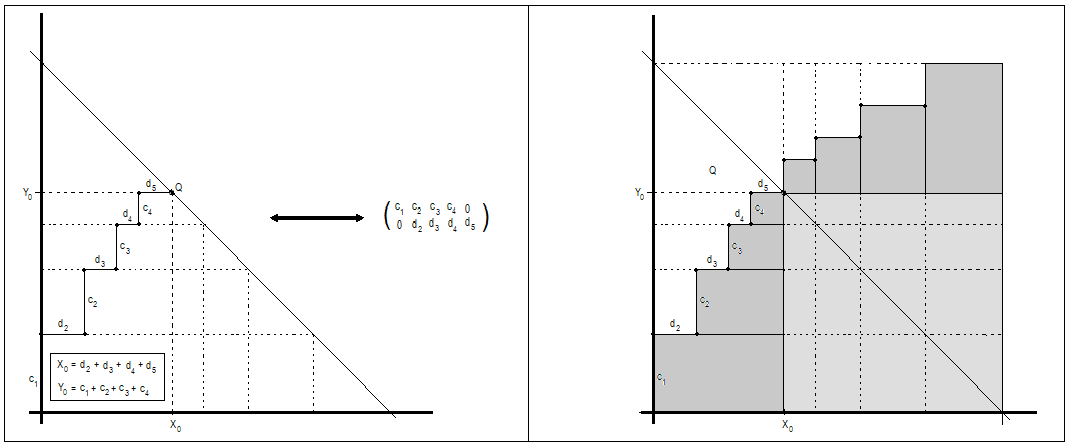}
\caption{}
\label{table4}
\end{figure}
Once we have the path, we can apply the Path Procedure to obtain the hooks and the partition into distinct odd parts, all greater than 1, of  the number $m$, which  is equal to the area of the figure, the sum of the parts (we have here a similar situation as  described in Figures \ref{fig_ex_partition_path} and \ref{fig_ex_partition_hooks}). 
 This area is also equal to the sum of the  area of the big square of size length equal to $y_{0}$ with twice the sum of the areas of the $(s-1)$ rectangles of areas $c_{1}\times x_{0}$, $c_{2}\times (x_{0}-d_{2})$, \ldots , $c_{s-2}\times (d_{s-1}+d_{s})$ and $c_{s-1}\times d_{s}$ respectively (see Figure \ref{table4} above). But due to the restrictions stated in \eqref{C}, we have that $x_{0}= c_{1},\; (x_{0}-d_{2})= c_{2},\; (x_{0}-d_{2}-d_{3}) = c_{3},\; \ldots,\; (d_{s-1}+ d_{s})=c_{s-2}$ and $d_{s}=c_{s-1}$.  Thus, in fact, these $(s-1)$ rectangles are $(s-1)$ squares, and we have
\begin{equation}\label{m-sq}
m = (c_{1}+\cdots + c_{s-1})^{2} + 2 (c_{1}^{2} + \cdots + c_{s-1}^{2}).
\end{equation}
Comparing this partition with the matrix $M$ we see immediatly the correspondence between them.
\end{proof}
The next corollary follows immediatly from the arguments given in the proof of Theorem \ref{Teo-sq}.
\begin{corollary}\label{Cor}
Let $\mathcal{M}_{0}$ be the subset of  $\mathcal{M}$ containing all matrices with the entry $d_{1}=0$, let $\mathcal{O}$ be the set of all partitions into distinct odd parts greater than 1 generated by the {\it Path Procedure}, and let $\mathcal{T}$  be the set of all $t$-squared partitions. There is an 1-1 correspondence between the sets $\mathcal{M}_{0}$,  $\mathcal{O}$ and  $\mathcal{T}$.
\end{corollary}

\begin{example}
	For the matrix
	$A=\left( \begin{array}{cccc} 2&1&1&0 \\ 0&1&0&1 \end{array} \right), $
	associated to the partition $(2,2,1,1)$ of $n=6$,  we have $P(A)=28$, given by the partition $28= 11+9+5+3$ (see Example \ref{ex_partition_hooks}). According to \eqref{m-sq} above and the entries of $A$, we also have $28=(2+1+1)^2 + 2 (2^{2}+1^{2} + 1^{2})$.
\end{example}

Let $M\in \mathcal{M}$,  written as   $M=\left ( \begin{array}{lllll} c_{1} & c_{2} & \cdots &c_{s-1}& 0 \\ d_{1}&d_{2} & \cdots &d_{s-1}& d_{s} 
\end{array}\right ),$  and define $\ell(M)$ to be equal to the sum of all entries of $M$, that is,

$$
 \ell(M)= \sum_{i=1}^{s} d_{i} \; + \;\sum_{j=1}^{s-1} c_{j}.
$$
 It follows from the conditions in \eqref{C}  that
\begin{equation}\label{part2}
\ell(M) = (c_{1}+d_{1}) + c_{1} + c_{2} + \cdots + c_{s-1}.
\end{equation}

\begin{lemma}
Let $m \in \N$, and assume that $m$ admits a $t$-squared partition. Let $M(m)\in  {\mathcal M}_{0}$ be the two-line matrix associated to $m$.
If $\ell(M(m)) \leq n$ then $m\leq n^{2}.$ 
\end{lemma}
\begin{proof}
It follows from Corollary \ref{Cor} and Remark \ref{m<n2}.
\end{proof}
The converse is not true as we have seen in  \eqref{conv}, for 
$$
m= 72 = P(A) < 9^{2}\:  \;\mbox{and} \;\;  \ell(A)=10.
$$
\begin{lemma}\label{PP}
Let $n\in \N$ and define $\mathcal{M}_{0}(n)$ to be the subset of $\mathcal{M}_{0}$ containing all  matrices $M$  such that $\ell(M) \leq n$. There is an 1-1 correspondence between the unrestricted partitions of $n$ into at least two parts and the elements of $\mathcal{M}_{0}(n)$.
\end{lemma}
\begin{proof} Initially observe that any matrix $A\in\mathcal{M}_{0}$ is always associated with a partition of $\ell(A)$ of the type $\ell(A) = \mu_{1} + \mu_{2} + \cdots + \mu_{s}$, with $\mu_{1}=\mu_{2} \geq \mu_{3}\cdots \geq \mu_{s}$.
Let $n=\lambda_{1} + \lambda_{2}+ \cdots + \lambda_{r}$ be an unrestricted partition of $n$, with $\lambda_{1} \geq  \lambda_{2}\geq  \cdots \geq  \lambda_{r} \geq 1$, and let $M\in \mathcal{M}$ be  the associated matrix.  If $ \lambda_{1}= \lambda_{2}$ then  $M\in\mathcal{M}_{0}(n)$.  Thus let us assume $ \lambda_{1}\neq \lambda_{2}$. In this case there exists a matrix $M^{*}\in\mathcal{M}_{0}$ corresponding to the partition $\lambda_{2} + \lambda_{2}+ \cdots + \lambda_{r}$, and  since $\ell(M^{*}) = n -(\lambda_{1} - \lambda_{2}) < n$, we have  $M^{*}\in\mathcal{M}_{0}(n)$. On the other hand, given a matrix $A\in\mathcal{M}_{0}(n)$, it is associated with a partitition  $\ell(A) = \mu_{2} + \mu_{2} +\mu_{3}+ \cdots + \mu_{s}$ and $\ell(A) \leq n$. Taking $ \delta = n- \ell(A)$, we will have  $(\mu_{2}+ \delta) + \mu_{2} +\mu_{3}+ \cdots + \mu_{s}$, a partition of $n$,  as desired (observe that the corresponding matrix to this partition will have $d_{1}=\delta$). This establishes the correspondence stated in the lemma.
\end{proof}
\section{The frequency of $m$}

It is easy to see that  a positive integer $m$ admits a $t$-{\it squared partition} if $m$ can be written as $m=b^2 + 2a$, and we can find a solution for the system
\begin{equation}\label{eq}
\left \{\begin{array}{ll}
b= x_{1}+\cdots + x_{t} \\
a= x_{1}^{2}+\cdots + x_{t}^{2},
\end{array} \right .
\end{equation}
with $x_{1},\ldots, x_{t}\in \N$.
\begin{lemma}\label{cond}
Let $m\in\N$  and suppose that $m$ admits  a $t$-{\it squared partition}. Then we can find $a,b\in \N$ such that  $m=b^2 + 2a$  with
$$
a\equiv b\pmod{2} \;\;\;\; \mbox{and}\;\;\; ta \geq b^{2} \geq  a \geq b.
$$
\end{lemma}
\begin{proof}
As seen above,  since $m$ admits  a $t$-{\it squared partition} there must exist  $a,b\in \N$ such that  $m=b^2 + 2a$, and  a solution $x_{1},\ldots,x_{t}\in \N$ for \eqref{eq}. The first statement  follows from the fact that  $x^{2}\equiv x\pmod{2}$, for any integer $x$. Since  $x_{1},\ldots,x_{t}\in \N$, we have $b^{2}\geq a\geq b$. The last inequality follows from the Cauchy-Schwarz inequality since
$$
b^{2}= (\sum_{i=1}^{t}x_{i})^{2} = (\sum_{i=1}^{t}x_{i}\cdot 1)^{2} \leq (\sum_{i=1}^{t}x_{i}^{2})(\sum_{i=1}^{t} 1^{2}) = ta.
$$
\end{proof}
\begin{lemma}
Let $m\in\N$. The integer $m$ admits  a $t$-{\it squared partition} only if $m\equiv 0 \;\mbox{or}\; 3 \pmod{4}$.
\end{lemma}
\begin{proof}
If $m$ admits  a $t$-squared partition then it can be written as $m=b^{2} + 2a$, and since  $a\equiv b\pmod{2}$, we have the conditions on $m$.
\end{proof}
\begin{corollary}
 Let $m\in \N$.  If $m\equiv 1 \;\mbox{or}\; 2 \pmod{4}$ then the  frequency $f(m)=0$.
\end{corollary}
\begin{theorem}
Let $m\in\N$, such that $m\in Im(P)$. Then $f(m)$ is equal to the number of  non-negative solutions $(c_{1},c_{2},\ldots, c_{b})$, assuming $c_{1}\geq c_{2}\geq \cdots \geq c_{b}\geq 0$,  of systems ot the type
\begin{equation}\label{Sys}
\left \{\begin{array}{l}
b= x_{1}+\cdots + x_{b} \\
a= x_{1}^{2}+\cdots + x_{b}^{2},
\end{array} \right .
\end{equation}
for any pair $a,b$ such that  $a\equiv b\pmod{2}$ and $m=b^{2}+ 2a$.
\end{theorem}
\begin{proof} Given a non-negative solution $(c_{1},c_{2},\ldots, c_{b})$ with  $c_{1}\geq c_{2}\geq \cdots \geq c_{b}\geq 0$, we may assume that for some $t\geq 1$ we have $c_{t}\neq 0$ and $c_{t+1}=\cdots =c_{b}=0$. This shows that $m$ admits a $t$-squared partition, and therefore there is a matrix $M \in \mathcal{M}_{0}$ associated with this partition such that $P(M)=m$. By Corollary \ref{Cor}, we have that this relation is 1-1.  Since, by definition,  $f(m)$ represents the number of matrices   $M \in \mathcal{M}_{0}$  such that $P(M)=m$, the result follows.
\end{proof}
Let $m\in\N$ and let $\mathbb{A}(m)$ be the set of all these  non-negative solutions $(c_{1},c_{2},\ldots, c_{b})$, given in the theorem above, hence $f(m) = |\mathbb{A}(m)|$.  As seen above, to each element of $\mathbb{A}(m)$  corresponds  a $t$-squared partition of $m$, which in turns corresponds to a matrix $M\in\mathcal{M}_{0}$. Observe that $\ell(M) = b + c_{1}$, according to \eqref{part2} and \eqref{Sys}, for in this case  $d_{1}=0$. Now let $\mathbb{B}(m)$ be the subset of  $\mathbb{A}(m)$ of the solutions with $ b + c_{1}\leq n$.  Now we are ready for our main result.
\begin{theorem}[Main Theorem]
Let $n\in\N$. Then
$$
p(n) = \sum_{m=1}^{n^{2}-1} |\mathbb{B}(m)| + 1.
$$
\end{theorem}
\begin{proof}  It follows from the comments above that each solution in $\mathbb{B}(m)$ is associated with only one matrix in $\mathcal{M}_{0}(n)$, for  $ b + c_{1}\leq n$, therefore 
$$
|\mathcal{M}_{0}(n)|  = \sum_{m=1}^{n^{2}-1} |\mathbb{B}(m)|,
$$
now the result follows from Lemma \ref{PP}, adding 1 for the only partition of $n$ in exactly one part.
\end{proof}
\section*{Acknowledgements}
The authors would like to express their gratitute to Professors Robson da Silva and Eduardo Brietzke for their valuable comments on an early draft of this paper.  This  paper  was  written  while  the  first  author  enjoyed the hospitality of the Universidade de Campinas in S\~{a}o Paulo-Brazil,  supported  by a grant from CNPq (Conselho Nacional de Desenvolvimento Cient\'ifico e Tecnol\'ogico)-Brazil


\end{document}